\documentclass[12pt]{amsart}
\usepackage{geometry}

\usepackage[utf8]{inputenc}
\usepackage[T1]{fontenc}
\usepackage{amsmath, amssymb, amsthm}
\usepackage{amsfonts}
\usepackage{mathtools}
\usepackage{graphicx}
\usepackage{xcolor}
\usepackage{enumerate}
\usepackage{hyperref, thmtools}
\usepackage{verbatim}
\usepackage{ulem}% strikethrough text using sout
\usepackage{cancel}
\usepackage{tikz, tikz-cd}
\usepackage{ytableau}
\usepackage{cite}

\definecolor{myblue}{RGB}{0,100,200}
\definecolor{myred}{RGB}{180,0,50}

\theoremstyle{plain}
\newtheorem{theorem}{Theorem}[section]
\newtheorem{proposition}[theorem]{Proposition}
\newtheorem{lemma}[theorem]{Lemma}
\newtheorem{corollary}[theorem]{Corollary}

\theoremstyle{definition}

\newtheorem{example}[theorem]{Example}
\newtheorem{mainthm}{Theorem}

\theoremstyle{remark}
\newtheorem{remark}[theorem]{Remark}

\numberwithin{equation}{section}

\newcommand{\CC}{\mathbb{C}}
\newcommand{\ZZ}{\mathbb{Z}}

\newcommand{\QQ}{\mathbb{Q}}
\newcommand{\cH}{\mathcal{H}}
\newcommand{\cF}{\mathcal{F}}

\newcommand{\cB}{\mathcal{B}}
\newcommand{\cL}{\mathcal{L}}
\newcommand{\cK}{\mathcal{K}}

\newcommand{\JW}{^{J}\! W}

\hypersetup{
	colorlinks=true,
	linkcolor=myblue,
	citecolor=myred,
	urlcolor=myblue
}
\title[Factorization of matrix of Kazhdan-Lusztig polynomials]{On factorization of matrix of \\ Kazhdan-Lusztig polynomials}
\author[Bhattacharya]{Aritra Bhattacharya}
\address{Beijing International Center for Mathematical Research, Peking University, Beijing 100871, China}
\email{matharitra@gmail.com}

\author[Mishra]{Ashish Mishra}
\address{The LNM Institute of Information Technology, Jaipur-302031, India}
\email{ashishmsr84@gmail.com, ashish.mishra@lnmiit.ac.in}

\author[Srivastava]{Shraddha Srivastava}
\address{Indian Institute of Technology Dharwad, 580011, India}
\email{maths.shraddha@gmail.com}

  

\date{\today}

\begin{document}
	
	\maketitle
	
	\begin{abstract}
		Let $\cH = \cH(W,S)$ be the Hecke algebra of the Coxeter system $(W,S)$ over $\ZZ[q^{\pm1}]$, where $W$ is the Weyl group of a symmetrizable Kac-Moody algebra. 
 In this paper, we show that the matrix of Kazhdan-Lusztig polynomials of $\cH$  factorizes into a product of $|S|$ many matrices, each of which has entries as polynomials in $q$ with nonnegative coefficients. To achieve this goal, we use hybrid basis $TC^J$ for  $J\subseteq S$ of $\cH$, defined by Grojnowski-Haiman. %This hybrid basis interpolates between the standard basis and Kazhdan-Lusztig basis of $\cH$. 
 The intermediate matrices in the aforementioned factorization turn out to be the transition matrices from $TC^J$-basis to $TC^I$-basis for $I\subset J$. Equivalently, these coefficients can be computed using a natural restriction map from $\cH$ to the parabolic Hecke algebra $\cH_J$. Moreover, following the ideas from Grojnowski-Haiman, we also give a geometric proof of the positivity of these coefficients.
 
 %For  $I\subseteq J$, it is natural to consider the coefficients of $TC^J$-basis in the expansion of  $TC^I$-basis. In this paper, we show that these coefficients are also polynomials in $q$ with nonnegative coefficients.  

	%For a Coxeter system $(W,S)$ and $J\subseteq S$, Grojnowski-Haiman defined a basis, called Hybrid basis $TC^J$, of the Hecke algebra $\cH$ of $W$ over the Laurent polynomial ring $\ZZ[q^{\pm1}]$. This hybrid basis interpolates between the standard basis and Kazhdan-Lusztig basis of $\cH$.	 The coefficients of $KL$-basis in the expansion of the standard basis are the classical KL polynomials. For  $I\subseteq J$, it is natural to consider the coefficients of $TC^J$-basis in the expansion of  $TC^I$-basis. In this paper, we show that these coefficients are also polynomials in $q$ with nonnegative coefficients.  This implies that the matrix of Kazhdan-Lusztig polynomials factorizes into a product of $|S|$ many matrices, each of which has entries as polynomials in $q$ with nonnegative coefficients. 

	\end{abstract}
	
	\section{Introduction}
	
%	The study of Kazhdan-Lusztig polynomials is a centerpiece in the interaction between Coxeter group combinatorics, geometry, and representation theory. 
For Hecke algebra of a Coxeter system, Kazhdan and Lusztig \cite{KL_original} defined Kazhdan-Lusztig polynomials as the integral polynomial coefficients for the change of basis from Kazhdan-Lusztig basis to standard basis and showed in \cite{KL80} that for finite and affine Weyl groups, these polynomials have nonnegative coefficients. Later on, Elias and Williamson \cite{EW14} showed the nonnegativity of these coefficients for arbitrary Coxeter systems. 

While considering the question of positivity of coefficients for the multiplication by Kazhdan-Lusztig basis element, Grojnowski and Haiman introduced hybrid basis, generalizing well-known positivity theorems of Springer \cite{Springer}, Lusztig \cite{Lusztig85}, and Dyer and Lehrer \cite{DL90}.  As hybrid basis interpolates between standard basis and Kazhdan-Lusztig basis, a natural question to explore in this more general setting is to study the positivity of coefficients for the change of basis between hybrid bases.  We answer  this question in \autoref{thm:mainB}, which provides us a factorization of the matrix of Kazhdan-Lusztig polynomials into a product of matrices, each of which has entries as polynomials with nonnegative coefficients, see \autoref{thm:mainC}.

%	The hybrid basis, as the name suggests, is a hybrid of standard basis and Kazhdan-Lusztig basis and interpolates between these two bases. 
	
	Consider a Coxeter system $(W,S)$, where $W$ is the Weyl group of a symmetrizable Kac-Moody algebra and $S$ is a set of simple reflections that generate $W$. The Hecke algebra $\cH$ is a $\ZZ[q^{\pm 1}]$-algebra that serves as a $q$-analogue of the group algebra $\ZZ[W]$ of $W$.
	For a subset $J \subseteq S$, one can consider the Coxeter subsystem $(W_J,J)$, where $W_J$ is the (parabolic) subgroup of $W$ generated by $J$ and let $\cH_J$ be its corresponding Hecke algebra. The Hecke algebras have a standard basis $\{T_w:w \in W\}$ indexed by elements of the Weyl group. Consider a natural restriction map $|_J: \cH \to \cH_J$ by defining on standard basis elements \begin{align*}
		\qquad T_w|_J = \begin{cases}
			T_w &\hbox{ if } w \in W_J,
			\\
			0 &\hbox{ otherwise,}
		\end{cases}
	\end{align*} and extending linearly to $\cH$.
	
%	There is a natural restriction map $\cH \to \cH_J$ given by sending the standard basis elements $T_w$ to itself or $0$ depending on whether $w$ is in $W_J$ or not. In other words, let $(T_w:w \in W)$ denote the standard basis of the Hecke algebra $\cH$, the restriction map $|_J:\cH \to \cH_J$ defined by \begin{align*}
%		&|_J: \cH \to \cH_J
%		\\
%		&\qquad T_w \mapsto \begin{cases}
%			T_w &\hbox{ if } w \in W_J,
%			\\
%			0 &\hbox{ otherwise,}
%		\end{cases}
%	\end{align*} and extend linearly to $\cH$. 
	
	Let $\{C_w:w \in W\}$ denote the Kazhdan-Lusztig basis of $\cH$ (these are denoted $C'_w$ in \cite{KL_original} by Kazhdan and Lusztig). Then one can ask whether the coefficients of $C_w|_J$ in the Kazhdan-Lusztig  basis of $\cH_J$ are nonnegative. In this paper, we confirm that this is indeed true. Let $W^J$ denote the minimal length left coset representatives of $W/W_J$. Our main result states that 
	\begin{mainthm}\label{thm:mainA}
		For $u \in W^J$ and $w \in W$, the coefficients of $(T_{u^{-1}}C_w)|_J$ with respect to the Kazhdan-Lusztig  basis of $\cH_J$ are polynomials in $q$ with nonnegative coefficients.
	\end{mainthm}
	
	 In what follows, the polynomials in \autoref{thm:mainA} are referred to as restriction coefficients. In Section \ref{sec:parKL} we show that the parabolic Kazhdan-Lusztig  polynomials corresponding to the sign representation appear as a special case of these. 
	 Positivity of coefficients in \autoref{thm:mainA} can be equivalently reinterpreted in terms of coefficients of change of basis for a hybrid basis $TC^J := \{TC^J_w: w \in W\}$ of $\cH$ defined by  Grojnowski and Haiman \cite{GrojnowskiHaiman}. This hybrid basis interpolates between the standard basis (when $J= \emptyset$) and the Kazhdan-Lusztig  basis (when $J=S$). Since the Kazhdan-Lusztig  basis expands positively in  the standard basis, one may ask a more general question: whether $TC^J$-basis expands positively in the $TC^I$-basis whenever $I \subseteq J \subseteq S$. We show that this is indeed the case in the following theorem. 

	\begin{mainthm}\label{thm:mainB} For $I \subseteq J \subseteq S$ and $w \in W$, the coefficients $h^{I,J}_{x,w}(q)$ in the expansion 
	   	\begin{equation}
		TC^J_w = \sum_{x \in W}^{} h^{I,J}_{x,w}(q) TC^I_x
	\end{equation} are polynomials in $q$ with nonnegative coefficients.
	   \end{mainthm} 
	  See Section \ref{sec:pos} for a proof of \autoref{thm:mainA} and \autoref{thm:mainB}. These are \autoref{th:resmainthm} and \autoref{th:mainpos}, respectively. 
	  \autoref{thm:mainB} immediately implies an interesting property of the matrix of Kazhdan-Lusztig polynomials:
	   
	   \begin{mainthm}\label{thm:mainC} For a chain of subsets 
	   	\begin{equation}\label{chain}
\emptyset = J_0 \subset J_1 \subset \cdots \subset J_k=S,
	   	\end{equation} the matrix of Kazhdan-Lusztig polynomials of $\cH$ factorizes into a product of $k$ many matrices, each of which has entries as polynomials in $q$ with nonnegative coefficients. In particular,  we can take the Chain \ref{chain} in such a manner that each subset $J_i$ has $i$ elements and $k=|S|.$
	   \end{mainthm}

	   We prove \autoref{lem:uz} and \autoref{lem:TuTCzI}, and use these along with \cite[Theorem 3.2]{GrojnowskiHaiman} to prove
	\autoref{thm:mainB}. Grojnowski and Haiman have used the techniques from perverse sheaves and mixed Hodge modules in order to prove  \cite[Theorem 3.2]{GrojnowskiHaiman}. In Section \ref{sec:posGH},  we give a geometric proof of \autoref{thm:mainA} that is essentially the arguments of \cite{GrojnowskiHaiman}. Since these geometric methods are available only for Weyl groups of symmetrizable Kac-Moody algebras, the corresponding statements of Theorems \ref{thm:mainA}, \ref{thm:mainB} and \ref{thm:mainC} for general Coxeter groups remain open. Computing explicitly the restriction coefficients in the case of dihedral groups in \autoref{prop:posDG} leads us to conclude that the positivity holds true for dihedral groups. In a subsequent work, we would like to study the corresponding positivity question for all Coxeter groups using the machinery of Soergel bimodules from \cite{EW14}. 
	
%	{\color{red} In the final section of this paper,  we also give a brief exposition of the part of the proof from their paper that is required for our purpose}.
	
	 We have computed various examples of the restriction coefficients in Section \ref{sec:rescof}. For Weyl group $W$ of type $A_n$, we compute the restriction coefficients $(T_{u^{-1}}C_w)|_{[n-1]}$, where $[n-1] = \{1,2,\ldots,n-1\}$ (for convenience of notation, we use $(T_{u^{-1}}C_w)|_{[n-1]}$ to mean $(T_{u^{-1}}C_w)|_{J[n-1]}$, where $J[n-1] =\{s_1, s_2, \ldots, s_{n-1}\}$), where $u$ runs over the set $W^{[n-1]} = \{1,s_n,s_{n-1}s_n,\ldots, s_2\cdots s_n, s_1s_2\cdots s_n\}$  and $w \in W$. We obtain explicit formulas for the restriction coefficients of $(T_{u^{-1}}C_w)|_{[n-1]}$ when $u = s_1s_2\cdots s_n, s_2s_3\cdots s_n$ and $s_3\cdots s_n$. 
	
	The organization of the paper is as follows. Section \ref{sec:HA} gives preliminaries for Hecke algebra of a Coxeter system. Section \ref{sec:res} contains restriction map and its properties. In Section \ref{sec:hybrid}, we define hybrid basis following \cite{GrojnowskiHaiman} and give connection between hybrid basis and the restriction map in \autoref{res:CtoTCisres}. Section \ref{sec:pos} deals with the proof of Theorems \ref{thm:mainA} and \ref{thm:mainB}. Section \ref{sec:parKL} shows that the parabolic Kazhdan-Lusztig  polynomials are a special case of restriction coefficients. In Section \ref{sec:rescof}, we compute various examples of restriction coefficients. In Section \ref{sec:posGH}, we give a geometric proof of \autoref{thm:mainA}.
	
	\subsection*{Acknowledgements} 
	We thank Arun Ram for helpful comments in the initial stages of the work. AB thanks Tao Gui for insightful discussions.

	\section{Preliminaries of Hecke algebras}\label{sec:HA}
	This section contains preliminaries of Hecke algebra of a Coxeter system. Our main reference for this section is \cite{IntrotoSoergelBimodules}.
	
	The \textit{Hecke algebra} $\cH = \cH(W)$ of a Coxeter system $(W,S)$ is the $\ZZ[q^{\pm1}]$-algebra generated by $\{T_s:s \in S\}$ with $T_s$ satisfying the braid relations and the quadratic relations: \begin{equation}
		T_s^2 = 1 + (q^{-1}-q)T_s, \qquad \hbox{ for all } s \in S.
	\end{equation} $$ $$
%	Then $$ T_s^{-1} = T_s + q-q^{-1}, \qquad \hbox{ for all } s\in S.$$
	Given a reduced expression $w= s_{i_1}\ldots s_{i_\ell}$ for $w \in W$, define $T_w = T_{s_{i_1}}\ldots T_{s_{i_\ell}}$. Since any two reduced expressions for $w$ are related by the braid relations $T_w$ is a well defined element of $\cH$. The collection $\{T_w: w \in W\}$ forms a basis of $\cH$ known as the \textit{standard basis}. 
%	The Hecke algebra $\cH$ has a basis indexed by the elements of the Coxeter group called the \textit{standard basis} $\{T_w: w \in W\}$, where if $w = s_{i_1}\ldots s_{i_\ell}$ is a reduced word for $w$ then $T_w = T_{s_{i_1}}\ldots T_{s_{i_\ell}}$.
%	If $sw>w$ then $T_sT_w = T_{sw}$ and if $sw<w$ then $$T_sT_w = T_sT_sT_{sw} = (1+(q^{-1}-q)T_s)T_{sw} = T_{sw} + (q^{-1}-q)T_{w}.$$ So 
Later, we also use the following multiplication rule	\begin{equation}\label{eq:TsTw}
		T_sT_w = \begin{cases}
			T_{sw} &\hbox{ if } sw>w,
			\\
			T_{sw} + (q^{-1}-q)T_{w} &\hbox{ if } sw<w.
		\end{cases}
	\end{equation}
		The \textit{bar involution} is a ring automorphism $\cH \to \cH$ given by $$ \overline{T_x} = (T_{x^{-1}})^{-1}, \quad \hbox{ for all } x \in W, \quad \hbox{ and } \quad \overline{q^{\pm1}} = q^{\mp1}.$$
	
	Recall the \textit{Bruhat order} is a partial order on $W$ defined by setting $u \leq v$ if any (or some) reduced expression for $v$ contains a subexpression for $u$.
	
	For each $w \in W$, there is a unique self-dual element $C_w \in \cH$ such that if \begin{equation}\label{eq:KLpolydef}
		C_w = \sum_{x \in W}^{} h_{x,w}(q) T_x,
	\end{equation} then \begin{enumerate}[(a)]
		\item $h_{x,w}(q) = 0$ if $x \nleq w$, and $h_{w,w}(q) = 1$,
		\item $h_{x,w}(q) \in q\ZZ[q]$ for $x<w$,
	\end{enumerate} where $<$ denotes the Bruhat order. In particular, the elements $C_w$ are unitriangular written in terms of standard basis with respect to the Bruhat order. Thus the set $\{C_w:w \in W\}$ forms a basis of $\cH$ known as the \textit{Kazhdan-Lusztig basis}. The Kazhdan-Lusztig  basis was introduced in \cite{KL_original}, where a different normalization is used and their $C'$ is our $C$. We follow the normalization used in \cite{IntrotoSoergelBimodules}, where $T_w$ is denoted $\delta_w$ and $C_w$ is denoted $b_w$.
	The coefficients $h_{x,w}(q) \in \ZZ[q]$ in \eqref{eq:KLpolydef} are known as the \textit{Kazhdan-Lusztig polynomials}. Unless there is ambiguity, we suppress $q$ in $h_{x,w}(q)$ and write $h_{x,w}$ for notational convenience.

	\begin{example}
		
	Two particular examples of Kazhdan-Lusztig  basis elements are as follows:
	\begin{enumerate}[(a)]
		\item For $s \in S$, $C_s = T_s + q.$
		
		\item Suppose $W$ is finite and $w_0$ the longest element in it. Then $$ C_{w_0} = \sum_{y \in W}^{} q^{\ell(w_0)-\ell(y)} T_y.$$
	\end{enumerate}
\end{example}

%	The expression $[q] h_{z,w}$ denotes the coefficient of $q$ in $h_{z,w}$. 
The Kazhdan-Lusztig  basis elements can be recursively constructed in the following way: Start with $C_1 = 1$.
	For $s \in S$ and $w \in W$ such that $ws>w$, we have \begin{equation}\label{eq:CwCs}
		C_{ws} = C_w C_s - \sum_{\substack{z<w \\ zs<z}}^{} \mu(z,w) C_z,
	\end{equation} where $\mu(z,w) :=$ the coefficient of $q$ in $h_{z,w}$.

%\subsection{$\omega$-involution and standard form}

The \textit{$\omega$-involution} is a ring anti-automorphism $\omega: \cH \to \cH$ given by $$ \omega(T_x) = T_{x}^{-1}, \qquad \hbox{ for all } x \in W, \qquad \hbox{ and } \qquad \overline{q^{\pm1}} = q^{\mp1}.$$

There is a standard $\ZZ$-bilinear form $(-,-):\cH \times \cH \to \ZZ[q^{\pm 1}]$ on $\cH$ defined by  $(\overline{T_x},T_y) = \delta_{x,y}$ for $x,y \in \cH$. The adjoint of multiplication by $a \in \cH$ under this form is given by multiplication by $\omega(a)$: \begin{equation}
	(a h, h') = (h, \omega(a) h') \qquad  \hbox{ for }  a, h, h' \in \cH.
\end{equation}

We recall the recursions for the Kazhdan-Lusztig  polynomials $h_{y,x}(q)$ following \cite{Casselmannnotes}. 
%For $y \leq x$, let $$\mu(y,x) = [q]h_{y,x}$$ denote the coefficient of $q$ in the Kazhdan-Lusztig polynomial $h_{y,x}$.

\begin{proposition}[Proposition 4.4 \cite{Casselmannnotes}]\label{prop:mainrecforKL}
	Suppose $s \in S$ such that $sx<x \mbox{ and }y \leq x$. Then
	\begin{align*}
		h_{y,x} = \begin{cases}
			\qquad\qquad\qquad 1 &\hbox{ if } y=x,
			\\
			h_{sy,sx}+q^c h_{y,sx} - \sum\limits_{\substack{w \\ y \leq w <sx \\ sw<w}}^{} \mu(w,sx) h_{y,w} &\hbox{ if } y \neq x,
		\end{cases}
	\end{align*} where $$ c = \begin{cases}
		-1 & \hbox{ if } sy<y,
		\\
		+1 & \hbox{ if } sy>y.
	\end{cases}$$
\end{proposition}

\begin{proposition}[Corollary 4.5 of \cite{Casselmannnotes}]\label{prop:Casselmannrec1}
	If $s \in S$ is such that $sx<x, sy<y<x, y \nleq sx$, then $$ h_{y,x} = h_{sy,sx}.$$
\end{proposition}

\begin{proposition}[Proposition 5.2 of \cite{Casselmannnotes}]\label{prop:Casselmannrecq}
	If $s\in S$ and $x,y \in W$ such that $sx<x \mbox{ and } sy>y$, then $h_{y,x} = qh_{sy,x}$. 
\end{proposition}

%\subsection{Notations}

{\bf Notations.} We use the characteristic function $\chi_{-}$ to which takes the value $1$ on a statement if it is true and the value $0$ on a statement if it is false. For example, $$\chi_{w \in W_J} = \begin{cases}
	1 &\hbox{ if } w \in W_J,
	\\
	0 &\hbox{ otherwise.}
\end{cases}$$ We also use the Kronecker delta function $$\delta_{x,y} = \begin{cases}
1 &\hbox{ if } x=y,
\\
0 &\hbox{ if } x\neq y.
\end{cases}$$
If $\cB$ is a basis, then for $b \in \cB$ we write $[b]h$ to denote the coefficient of $b$ in the $\cB$ expansion of $h$, i.e, $$ h = \sum_{b \in \cB }^{} ([b]h) b.$$
	
% \newpage
\section{The restriction map}\label{sec:res}
In this section, we define a restriction map and study its properties. 

Let $J \subseteq S$ and $W_J = \langle J \rangle$ is the parabolic subgroup of $W$ corresponding to $J$, with corresponding Hecke algebra $\cH_J = \cH(W_J)$. Using the standard basis $\{T_w: w \in W\}$ of $\cH$, we  define a restriction map  \begin{align*}
		&|_J: \cH \to \cH_J
		\\
		&\qquad T_w \mapsto \begin{cases}
			T_w &\hbox{ if } w \in W_J,
			\\
			0 &\hbox{ otherwise,}
		\end{cases}
	\end{align*} and extend linearly to $\cH$. 
	
	\begin{proposition}\label{res:resH_Jlinear}
		The restriction map $|_J$ is both left and right $\cH_J$-linear: For $g \in \cH_J$ and $h \in \cH$,	\begin{equation}
			(hg)|_J = h|_J \cdot g \qquad \hbox{ and } \qquad
			(gh)|_J = g\cdot h|_J.
		\end{equation}
	\end{proposition}
	
	\begin{proof} We prove the left sided $\cH_J$-linearity of the restriction map $|_J$ by induction on the length of the elements of  $W_J$. 
		
\noindent	Let $s \in J$ and $w \in W$.	If $sw>w$, then $T_sT_w = T_{sw}$, which gives \begin{align*}
			(T_sT_w)|_J = T_{sw}|_J = \chi_{w \in W_J} T_{sw} = T_s \chi_{w \in W_J} T_w = T_s T_w|_J.
		\end{align*} If $sw<w$, then by \eqref{eq:TsTw}, we have \begin{align*}
		(T_sT_w)|_J = (T_{sw} + (q^{-1}-q)T_w)|_J = \chi_{w \in W_J} (T_{sw} + (q^{-1}-q) T_w) = T_s \chi_{w \in W_J} T_w = T_s T_w|_J. 
		\end{align*} 
%	Then for $s \in J$ and $w \in W$, $$ (T_s T_w)|_J = T_s T_w|_J,$$
	
\noindent	 By linearity, we conclude that for $s\in J$, $h \in \cH$, $$ (T_sh)|_J = T_s h|_J.$$
	
	\noindent Now by induction on length of $v \in W_J$, we prove that for  $h \in \cH$, $$(T_vh)|_J = T_vh|_J.$$ Suppose $\ell(sv) = \ell(s) + \ell(v)$ with $s \in J, h \in \cH$. Then \begin{align*}
			(T_{sv}h)|_J = (T_s T_v h)|_J = T_s (T_vh)|_J = T_s T_v h|_J = T_{sv} h|_J.
		\end{align*} This proves the left sided $\cH_J$-linearity of the restriction map $|_J$.
		Similarly, the right sided $\cH_J$-linearity is obtained.
	\end{proof}
	
	Using $\cH_J$-linearity, we compute the following examples.

	\begin{example}
		Since $C_s = T_s+q$ for $s \in S$, the restriction $$ C_s|_J = \begin{cases}
			C_s &\hbox{ if } s \in J,
			\\
			q &\hbox{ if } s \notin J.
		\end{cases}$$
		Consider $\{s_1, s_2, \ldots, s_k\} \subseteq S$ and $J=\{s_1, s_2, \ldots, s_{j}\}$ for $j \leq k$. 
		Then \begin{align*}
			(C_{s_1}C_{s_2}\cdots C_{s_k})|_{J} = q^{k-j}C_{s_1}\cdots C_{s_j} = (C_{s_{j+1}}C_{s_{j+2}}\cdots C_{s_k}C_{s_1}\cdots C_{s_j})|_{J}.
		\end{align*}
%		Although we can write the corresponding equality for an arbitrary subset of $\{1,2,\ldots,k\}$, we have considered $[j] $ for the sake of simplicity.
	\end{example}
	
	\begin{example}
	Since ${}^JW \cap W_J = \{1\}$, we get for $u \in W^J$ and $s \in J$,	$(T_{u^{-1}}C_s)|_J = \delta_{u,1}C_s$.
\end{example}
	
	Suppose $u,u' \in W^J$ and $v \in W_J$, then by \autoref{res:resH_Jlinear}, \begin{align*}
		(T_{u^{-1}}T_{u'v})|_J = (T_{u^{-1}}T_{u'}\cdot T_v)|_J = (T_{u^{-1}}T_{u'})|_J \cdot T_v.
	\end{align*}

	\begin{lemma}\label{lem:resorthogonality}
		For $u,u' \in W^J$, $$(T_{u^{-1}}T_{u'})|_J = \delta_{u,u'}.$$
	\end{lemma}
	
	\begin{proof}

		The coefficient of $T_v$ for $v \in W_J$ in $(T_{u^{-1}}T_{u'})|_J$ equals 
	\begin{align*}
			(\overline{T}_v, T_{u^{-1}}T_{u'}) = (\overline{T}_v, \omega(\overline{T}_{u}) T_{u'}) =  (\overline{T}_u\overline{T}_v,  T_{u'}) = (\overline{T}_{uv}, T_{u'}) = \delta_{u,u'} \delta_{1,v}.
		\end{align*} 
	\end{proof}

For $u \in W^J, w \in W$, by \autoref{lem:resorthogonality}, \begin{equation}\label{eq:rescoeffgenT}
	(T_{u^{-1}}C_w)|_J = \sum_{y \leq w}^{} h_{y,w}(q)(T_{u^{-1}}T_y)|_J = \sum_{v \in W_J}^{} h_{uv,w}(q)T_v.
\end{equation}

\begin{lemma}\label{lem:TuinvCwressupport}
	Let $u \in W^J$ and $w \in W.$ Then, for the parabolic factorization $w=w_J w^J$,
	$$ (T_{u^{-1}}C_w)|_J = 0 \hbox{ if } u \nleq w^J.$$
\end{lemma}

\begin{proof}
	By \eqref{eq:rescoeffgenT}, if $T_v$ appears in the $T$-expansion of $(T_{u^{-1}}C_w)|_J$ then we must have $uv \leq w$. By \cite[Proposition 2.5.1]{BjornerBrenti} parabolic factorization respects Bruhat order for quotients. Hence $u \leq w^J$. 
\end{proof}	

%{\color{red}\begin{example}Follows from $\mathcal{H}_J$-linearity and hence moved above. 
%					Using \begin{equation}\label{eq:TxCs}
%	T_xC_s = \begin{cases}
%		T_{xs} + qT_x & \hbox{ if } xs>x,
%		\\
%		T_{xs} + q^{-1}T_x & \hbox{ if } xs<x
%	\end{cases} 
%\end{equation} and $\JW \cap W_J = \{1\}$, we get for $u \in W^J$ and $s \in J$, \begin{align*}
%	(T_{u^{-1}}C_s)|_J = \delta_{u,1}C_s|_J.
%\end{align*}
%\end{example}}
\noindent
For $u \in W^J, w \in W$, define $h^J_{uv,w}(q)$ by \begin{equation}\label{eq:coeff}
	(T_{u^{-1}}C_w)|_J = \sum_{v \in W_J}^{} h^J_{uv,w}(q)C_v.
\end{equation} 

\noindent
Combining \eqref{eq:rescoeffgenT}, \eqref{eq:coeff} and \eqref{eq:KLpolydef} we get $$ h_{uy,w}(q) = \sum_{v \in W_J}^{} h^J_{uv,w}(q)h_{y,v}(q) \quad \hbox{ for } y \in W_J.$$ Since the matrix $(h_{y,v}(q))_{y,v \in W_J}$ is unitriangular, this shows that $$ h^J_{uv,w}(q) \in \ZZ[q]$$ for $u \in W^J, v \in W_J$ and $w \in W$. Note that this also implies $h^J_{uv,w}(q) = 0$ if $uv \nleq w$.

The following proposition gives a condition for vanishing of $ h^J_{uv,w}(q).$

\begin{proposition}
	Let $w \in W, u \in W^J, v \in W_J$ and $s \in J$ be such that $vs>v, ws<w$. Then $ h^J_{uv,w}(q) = 0.$
\end{proposition}

\begin{proof}
	Since \begin{align*}
		&(q+q^{-1})(T_{u^{-1}}C_w)|_J = (T_{u^{-1}}C_w C_s)|_J = (T_{u^{-1}}C_w)|_J C_s = \sum_{v' \in W_J}^{} h^J_{uv',w}(q)C_{v'}C_s 
		\\
		&= \sum_{\substack{v' \in W_J \\ v's>v'}}^{} h^J_{uv',w}(q)\bigg( C_{v's} + \sum_{\substack{z<v' \\ zs<z}}^{} \mu(z,v')C_z \bigg) + \sum_{\substack{v' \in W_J \\ v's<v'}}^{} h^J_{uv',w}(q) \cdot (q+q^{-1})C_{v'},
	\end{align*} $(T_{u^{-1}}C_w)|_J$ only has those $C_v$ with $vs<v$ appearing. 
\end{proof}

%%%%%%%%%%%%%%%%%%%%%%%%%%%%%%%%%%%%%%%%%%%%%%%%%%%%%
	% \newpage
	\section{The hybrid basis}\label{sec:hybrid}
	This section defines hybrid basis of Hecke algebra $\cH$. Also we establish a connection between hybrid basis and restriction map in \autoref{res:CtoTCisres}. Furthermore, we show that the parabolic Kazhdan-Lusztig  polynomials are a special case of restriction coefficients.
	
%		\subsection{Definition of hybrid basis} 
		Let $J \subseteq S$ be a subset of the set $S$ of simple reflections. We recall the definition of \textit{hybrid basis} $TC^J:=\{TC^J_w: w \in W\}$ and $CT^{J}:=\{CT^J_w: w \in W\}$ of the Hecke algebra $\cH$ as defined by Grojnowski and Haiman in \cite{GrojnowskiHaiman}. Let $W_J$ be the parabolic subgroup of $W$ generated by $J$. Let $W^J$ be the minimal length left coset representatives in $W/W_J$ and $\JW$ the minimal length right coset representatives in $W_J\backslash W$. For $w \in W$, there are parabolic factorizations \begin{align*}
		w = w^J\cdot w_J \qquad \hbox{ and } \qquad w = _J\!\!w \cdot^{J}\!\! w 
	\end{align*} with $w_J, _J\!w \in W_J$ and $w^J \in W^J, ^{J}\!\! w \in \JW$. \cite[Proposition 2.4.4]{BjornerBrenti} says that such a factorization exists and is unique, with $\ell(w) = \ell(w^J) + \ell(w_J) = \ell(_Jw)+\ell(^{J\!} w)$. Define \begin{equation}
		TC^J_w = T_{w^J} C_{w_J} \qquad \hbox{ and } \qquad CT^J_w = C_{_Jw}T_{^{J\!} w}.
	\end{equation} 
On the extreme ends, we recover the standard and Kazhdan-Lusztig bases of $\cH$: when $J = \emptyset$, $TC^\emptyset_w = T_w$ and when $J = S$, $TC^S_w = C_w$. 	

	There exists a ring anti-automorphism $\Psi: \cH \rightarrow \cH$ such that $\Psi(T_w)=T_{w^{-1}}$ and $\Psi$ fixes $q$. Further, $\Psi$ commutes with bar involution and therefore, $\Psi(C_w)=C_{w^{-1}}$. Thus, \begin{equation}\label{eq:leftrightconv}
		\Psi(TC^J_w) = \Psi(C_{w_J}) \Psi(T_{w^J}) = C_{w_J^{-1}}T_{{(w^J)}^{-1}} = CT^J_{w^{-1}}.
	\end{equation} 
	Because of this, we will get similar results  for both types of hybrid bases. It will be enough to consider one or the other, depending on the context.  
	%We  can move between mostly focus on one or the other in the discussions below.

%	{\color{red}  \begin{remark} Relevance?
%		The definition of hybrid basis may be restated as $$ TC^J_w = TC^\emptyset_{w^J} TC^{S}_{w_J}.$$
%	\end{remark}}

	Now we show that $TC^J$ is a basis of $\cH$.
	
	\begin{proposition}
		Let $\cH_J$ be the Hecke algebra of the Coxeter subsystem $(W_J,J)$ generated by $J \subseteq S$. Let $\cB = \{b_w: w \in W_J\}$ be a basis for $\cH_J$. Then the set $\{T_ub_v \,|\, u \in W^J, v \in W_J \}$ is a basis for $\cH$.
	\end{proposition}
	\begin{proof}
		To show that $\{T_ub_v \,|\, u \in W^J, v \in W_J \}$ is a basis, it is enough to show that each $T_w: w\in W$ can be uniquely written as a linear combination of $T_ub_v'$s. By uniqueness of parabolic factorization, we have $T_w = T_u T_v$ for unique $u \in W^J, v \in W_J$. Also, $T_v \in \cH_J$ and hence, can be expressed uniquely as a linear combination of $\{b_v: v \in W_J\}$.
			\end{proof}
		
%		By expanding each $b_v$ as a linear combination of $T_w: w\in W_J$, and using linear independence of $\{T_w: w\in W\}$, we see that the only linear dependence relation possible among the $T_ub_v$s for $u \in W^J, v \in W_J$ is when $u$ is fixed. But in that case a linear dependence would imply a linear dependence among the $b_v$s (as $T_u$ is invertible), which is also impossible, since $\cB$ is a basis. Thus the set $\{T_ub_v \,|\, u \in W^J, v \in W_J \}$ is linearly independent. 
%		By uniqueness of parabolic factorization any $T_w: w\in W$ is a linear combination of $\{T_ub_v: v \in W_J\}$ for a fixed $u \in W^J$, but that combination is unique since $\cB$ is a basis.

By \cite[Proposition 2.4.4, (2.11),(2.12)]{BjornerBrenti}, if $u \in W^J$ and $v \in W_J$ then $\ell(uv) = \ell(u)+\ell(v)$, so $T_{u}T_v = T_{uv}$. Since $\ell(uv) = \ell(u)+\ell(v)$, a reduced expression for $uv$ is obtained by concatenating a reduced expression for $u$ and a reduced expression for $v$. Then by subword criteria, if $x<v$ then  $ux < uv = w$. 
	Then \begin{align*}
		TC^J_w &= T_uC_v = T_u(T_v+ \sum_{x<v} h_{x,v}(q) T_x ) = T_{uv} + \sum_{x<v} h_{x,v}(q) T_{ux}
		\\
		&= T_{w} + \sum_{x<v} h_{x,v}(q) T_{ux},
	\end{align*} so $TC^J_w$ is unitriangular with respect to the Bruhat order when expressed in terms of the standard basis.

The following proposition explains a connection between $TC^J$-basis and the restriction map defined in Section~\ref{sec:res}.

	\begin{proposition}\label{res:CtoTCisres} 	For $u \in W^J$, $v \in W_J$ and $h \in \cH$, the coefficient of $TC^J_{u v}$ in the $TC^J$-expansion of $h$ equals the coefficient of $C_v$ in the Kazhdan-Lusztig -expansion of $(T_{u^{-1}}h)|_J$: \begin{equation}\label{eq:res:CtoTCisres}
			[TC^J_{u v}]h = [C_v](T_{u^{-1}}h)|_J.
		\end{equation} 
\end{proposition}

\begin{proof}
	Suppose $$ h = \sum_{\substack{u'\in W^J \\ v \in W_J}}^{} b_{u'v}(q)TC^J_{u' v}.$$ Then using  $\cH_{J}$-linearity of the restriction map and \autoref{lem:resorthogonality}, \begin{align*}
		(T_{u^{-1}}h)|_J = \sum_{\substack{u' \in W^J \\ v \in W_J}}^{} b_{u' v}(q)(T_{u^{-1}}T_{u'}C_v)|_J = \sum_{v \in W_J}^{} b_{uv}(q)C_v.
	\end{align*} Thus, $[C_v](T_{u^{-1}}h)|_J = b_{uv}(q) = [TC^J_{uv}]h$.
\end{proof}

Below we recall the following result from~\cite[Corollary 3.9]{GrojnowskiHaiman}, which is the source of our main results.
	
	\begin{theorem}\label{th:GH}
Let $(W,S)$ be the Coxeter system of a Weyl group of a symmetrizable Kac-Moody algebra, and $\cH$ be its Hecke algebra and $J\subseteq S$. For $w,u\in W$, the product $C_w TC^J_u$ expands positively in the $TC^J$-basis:
		$$ C_w TC^J_u = \sum_{v \in W} d^w_{u,v}(q)TC^J_v$$
		with $d^w_{u,v}(q) \in \ZZ_{\geq 0}[q^{\pm 1}]$.
\end{theorem}	

	\section{Proofs of main results}\label{sec:pos}
	In this section, we give a proof of \autoref{thm:mainB} (restated as  \autoref{th:mainpos}) and a proof of \autoref{thm:mainA} (restated as \autoref{th:resmainthm}). 
	\begin{theorem}\label{th:mainpos}
		Let $(W,S)$ be the Coxeter system of a Weyl group of a symmetrizable Kac-Moody algebra, and $\cH$ be its Hecke algebra. If $I \subseteq J \subseteq S$, then for $w \in W$, $TC^J_w$ expands positively in $TC^I$-basis: \begin{equation}\label{eq:hIJdef}
			TC^J_w = \sum_{x \in W}^{} h^{I,J}_{x,w}(q) TC^I_x,
		\end{equation} with coefficients $h^{I,J}_{x,w}(q) \in \ZZ_{\geq 0}[q]$.
	\end{theorem}
We defer giving a proof of the above theorem until the end of this section. Now we state and prove its corollary.
	
	\begin{corollary}\label{th:resmainthm}
		Let $(W,S)$ be the Coxeter system of a Weyl group of a symmetrizable Kac-Moody algebra, and $\cH$ be its Hecke algebra. Let $J \subseteq S$ and $\cH_J$ be the Hecke algebra of the Coxeter system $(W_J,J)$. For $u \in W^J$ and $w \in W$, the element $(T_{u^{-1}}C_w)|_J \in \cH_J$ when expressed in the Kazhdan-Lusztig  basis of $\cH_J$ has coefficients in $\ZZ_{\geq 0}[q]$: \begin{equation}
			(T_{u^{-1}}C_w)|_J = \sum_{v \in W_J}^{} h^J_{uv,w}(q)C_v,
		\end{equation} with $h^J_{uv,w}(q) \in \ZZ_{\geq 0}[q]$.
	\end{corollary}
	
	\begin{proof}
		By \autoref{res:CtoTCisres},  \begin{equation}\label{eq:hJeqhJS}
			h^J_{uv,w}(q) = [C_v](T_{u^{-1}}C_w)|_J = [TC^J_{uv}]C_w = h^{J,S}_{uv,w}(q).
		\end{equation}
	%	the result follows \cite[Theorem 3.2]{GrojnowskiHaiman}.
		
		%as the last polynomial is equal to $c^{w}_{uv}(q)$ in \cite[Theorem 3.2]{GrojnowskiHaiman}
		\end{proof}
	%\subsection{Examples}
Below we give two examples illustrating \autoref{th:mainpos}.
	\begin{example}
		When $J = \{1\}$ and $I = \emptyset$, then $W_J = \{1,s_1\}$, $W^J = \{u\in W \, |\, us_1 > u \}$. If $w \in W$, $ws_1 < w$ then \begin{align*}
			TC^{[1]}_w = T_{ws_1}C_{s_1} = T_{ws_1}(T_{s_1} + q) = T_{w} + qT_{ws_1}.
		\end{align*} If $w \in W$, $ws_1 > w$, then $TC^{[1]}_w = T_w.$
	\end{example}
	
	%When $J = \{1\}$ and $I = \emptyset$, then $W_J = \{1,s_1\}$, $W^J = \{u\in W \, |\, us_1 > u \}$. If $w \in W$, $ws_1 < w$ then \begin{align*}
		%	TC^{[1]}_w = T_{ws_1}C_{s_1} = T_{ws_1}(T_{s_1} + q) = T_{w} + qT_{ws_1}.
		%\end{align*} If $w \in W$, $ws_1 > w$, then $TC^{[1]}_w = T_w.$
		
		\begin{example}
			For type A, $J = \{1,2\}$ and $I = \{1\}$, the parabolic factorizations for $I = \{1\}$ are \begin{align*}
				1 = 1\cdot 1, s_1 = 1\cdot s_1, s_2 = s_2 \cdot 1, s_1s_2 = s_1s_2\cdot 1, s_2s_1 = s_2 \cdot s_1, s_1s_2s_1 = s_1s_2 \cdot s_1.
			\end{align*}
			
			Let $u \in W^{\{1,2\}}$. Then \begin{align*}
				&TC^{\{1,2\}}_{u} = T_u = TC^{\{1\}}_u,
				\\
				&TC^{\{1,2\}}_{us_1} = T_uC_{s_1} = TC^{\{1\}}_{us_1},
				\\
				&TC^{\{1,2\}}_{us_2} = T_uC_{s_2} = T_{us_2}+qT_u = TC^{\{1\}}_{us_2}+qTC^{\{1\}}_u,
				\\
				&TC^{\{1,2\}}_{us_1s_2} = T_uC_{s_1}C_{s_2} = T_{us_1s_2}+q(T_{us_1}+T_{us_2})+q^2T_u = TC^{\{1\}}_{us_1s_2}+qTC^{\{1\}}_{us_2}+qTC^{\{1\}}_{us_1},
				\\
				&TC^{\{1,2\}}_{us_2s_1} = T_uC_{s_2}C_{s_1} = T_u(T_{s_2}+q)C_{s_1} = T_{us_2}C_{s_1}+qT_{u}C_{s_1} = TC^{\{1\}}_{us_2s_1}+qTC^{\{1\}}_{us_1},
				\\
				&TC^{\{1,2\}}_{us_2s_1s_2} =  T_{us_1s_2s_1}+q(T_{us_1s_2}+T_{us_2s_1})+q^2(T_{us_1}+T_{us_2})+q^3T_u
				\\
				&\qquad\qquad = TC^{\{1\}}_{us_1s_2s_1}+qTC^{\{1\}}_{us_2s_1}+q^2TC^{\{1\}}_{us_1},
			\end{align*}
			since \begin{align*}
				&TC^{\{1\}}_{us_1} = T_uC_{s_1} = T_{us_1} + qT_{u},
				\\
				&TC^{\{1\}}_{us_1s_2s_1} = T_{us_1s_2}C_{s_1} = T_{us_1s_2s_1}+qT_{us_1s_2},
				\\
				&TC^{\{1\}}_{us_2s_1} = T_{us_2}C_{s_1} = T_{us_2s_1}+qT_{us_2}.
			\end{align*}
		\end{example}
		
			%\subsection{} 
			Now we prove two lemmas which will lead us to a proof of \autoref{th:mainpos}.
		
		\begin{lemma}\label{lem:uz}
			Let $I \subseteq J \subseteq S$ and $u \in W^J, z \in W^I \cap W_J$. Then $uz \in W^I$.
		\end{lemma}
		
		\begin{proof}
			We must show  that
			 $\ell(uzs)>\ell(uz) \text{ for every }s\in I.$
			
			Fix \(s\in I\). Since $z \in W_J$ and $s \in W_J$, $zs \in W_J$. Because $u \in W^J$ and $zs \in W_J$, $$\ell(uzs) = \ell(u) + \ell(zs).$$  
			Since \(z\in W^{I}\), no simple reflection in \(I\) is a right descent of \(z\); hence  
			\[\ell(zs)=\ell(z)+1.\]
			Also $u \in W^J$ and $z \in W_J$ imply that $ \ell(u) + \ell(z) = \ell(uz).$
			Therefore  
			\[\ell(uzs)=\ell(u)+\ell(zs)=\ell(u)+\ell(z)+1=\ell(uz)+1>\ell(uz).\]
			%As \(s\in I\) is arbitrary, \(uz\) satisfies the defining condition for \(W^{I}\). 
			\end{proof}
		
		\begin{lemma}\label{lem:TuTCzI}
			If $u \in W^J, z \in W_J$ and $I \subseteq J$ then $T_u TC^I_z = TC^{I}_{uz}$
		\end{lemma}
		\begin{proof}
		
			Suppose $z = z^I \cdot z_I$ is the parabolic factorization of $z$ with respect to $I$. Then $z^I \in W^I \cap W_J$. So by \autoref{lem:uz}, $uz^I \in W^I$. Hence, $uz = uz^I \cdot z_I$ is a parabolic factorization of $uz$ with respect to $I$. Also $u \in W^J$ and $z^I \in W_J$ imply that $\ell(uz^I) = \ell(u)+\ell(z^I)$.	Then $TC^I_{uz} = T_{uz^I} C_{z_I} = T_u T_{z^I} C_{z_I}=T_uTC^I_z$.
		\end{proof}
	
	\begin{proof}[Proof of \autoref{th:mainpos}]	
		Suppose $w = uv$ with $u \in W^{J}$ and $v \in W_{J}$. Then $TC^{J}_{w} = T_u C_v$. 
		Using \eqref{eq:hIJdef} and \eqref{eq:hJeqhJS} write $$ C_v = \sum_{x \in W}^{} h^I_{x,v}(q) TC^I_x.$$  Since $v \in W_{J}$ and the sum can only run over $x \leq v$ in Bruhat order, the sum must be over $x \in W_{J}$. Then $$TC^{J}_{w} = T_u C_v =  T_u\sum_{x \in W}^{} h^I_{x,v}(q) TC^I_x = \sum_{x \in W_{J}} h^I_{x,v}(q) TC^{I}_{ux}$$ by \autoref{lem:TuTCzI}.
		Therefore,
		
\begin{equation}\label{eq:coeffreductiontoCtoTC}
			[TC^{I}_{uv}]TC^{J}_{u'v'} = [TC^{I}_v]C_{v'} \delta_{u,u'}  \quad \hbox{ for } u,u' \in W^J,\ v,v' \in W_J.\end{equation}
		% where we use the notation $[TC^I_{x}]h$ to mean the coefficient of $TC^I_x$ in the expansion of $h \in \cH$ in the $TC^I$-basis.
By \autoref{th:GH}, we have $[TC^{I}_v]C_{v'} \in \ZZ_{\geq 0}[q]$. This proves \autoref{th:mainpos}. %as a special case of their general result.
\end{proof}		

In fact, the proofs of \autoref{th:mainpos} and \autoref{th:resmainthm}  imply that  \autoref{th:mainpos} and \autoref{th:resmainthm} are  equivalent.
		\begin{remark}
			Since \autoref{th:GH} deals only with Weyl group of a symmetrizable Kac-Moody algebra, we cannot say anything right away about general Coxeter groups from this proof.
		\end{remark}
Following ideas from \cite{GrojnowskiHaiman}, we give a geometric proof of \autoref{th:resmainthm} in Section \ref{sec:posGH}. % with some minor modifications.

\section{Parabolic Kazhdan-Lusztig polynomials}\label{sec:parKL}

In this section we show that the parabolic Kazhdan-Lusztig polynomials associated to sign representations are special cases of restriction coefficients.
These polynomials were defined by Deodhar in~\cite{Deodhar}, below we recall them using normalization from~\cite{Soergel}. As shown in \cite{GrojnowskiHaiman}, the parabolic Kazhdan-Lusztig positivity can be derived from their main result \autoref{th:GH}. In fact, it may be deduced from our weaker statement \autoref{th:mainpos} or \autoref{th:resmainthm} as we will show below that the parabolic Kazhdan-Lusztig polynomials are special cases of restriction coefficients. 

Consider the sign representation of $\cH_J$: $$\epsilon(T_w)=(-q)^{\ell(w)} \text{ for } w\in W_J.$$
Then the induced representation $\text{Ind}_{\cH_J}^{\cH}(\epsilon)$ is isomorphic to the $\cH$-module $\cH e_J^{-}$, where $e_J^{-}$ satisfies 
\begin{equation}\label{eq:Ts}
	(T_s+q)e_J^{-}=0, \text{ for } s\in J.
\end{equation}
The sets $\{T_u e_J^{-}\mid u\in W^J\}$ and $\{C_u e_J^{-}\mid u\in W^J\}$ are bases of 
module $\cH e_J^{-}$. The parabolic Kazhdan-Lusztig polynomials $P^{J^-}_{u,u'}$ for $u,u' \in W^J$ are defined to be the change of basis coefficients 
$$C_{u'} e_J^{-} = \sum_{u\in W^J} P^{J^{-}}_{u,u'}(q)  T_u e_J^{-}.$$
%where $P^{J^{-}}_{u,w}(q) $ are called parabolic KL polynomials and  they have the following properties:
%\begin{enumerate}[(a)]
%	\item $P^{J^{-}}_{u,w}(q) = 0$ if $u \nleq w$, and $P^{J^{-}}_{w,w}(q) = 1$,
%	\item $P^{J^{-}}_{u,w}(q) \in q\ZZ[q]$ for $u<w$.
%\end{enumerate} 
\begin{lemma} For $v \in W_J\setminus \{1\}$,
	\begin{equation}\label{eq:pC}
		C_{v} e_{J}^-= 0. %\qquad \hbox{ for }  v \in W_J\setminus \{1\}. 
\end{equation}
\end{lemma}
\begin{proof}
	Since for $s \in S$, $C_s = T_s+q$, \eqref{eq:Ts} says that $$C_s e_J^- = 0 \qquad \hbox{ for } s \in J.$$ Suppose that $C_z e_J^- = 0$ for $z\leq v \in W_J$ with $z \neq 1$. By \eqref{eq:CwCs} \begin{align*}
		C_{vs}e_J^- = C_vC_se_J^- - \sum_{\substack{z<v \\ zs<z}}^{} \mu(z,v) C_ze_J^- = 0
	\end{align*} since all $z \neq 1$ in the last sum.
\end{proof}

\noindent
Then for $w \in W$, \begin{equation}\label{eq:pTC}
	TC_{w} e_{J}^-=\begin{cases}
		T_{w}e_{J}^- &\text{ if } w\in W^J,\\
		0 & \text{ otherwise}.
	\end{cases}
\end{equation}
%
%where the first case is immediate and the second case can be proved using induction as follows. Let $w'\notin W^J$ and let $w'=u\cdot v$ be its parabolic factorization. If $v=s\in J$, then $T_{w'} e_{J}^- = T_u C_se_{J^-}=0$, due to~\eqref{eq:Ts}. For $v\in W_J$ and $s\in J$ such that $vs>v$, we can use recursion~\eqref{eq:recursion} to conclude that $T_u C_{vs}e_{J^-}=0$

\noindent
By \eqref{eq:hIJdef} and \eqref{eq:hJeqhJS} $$ C_w = \sum_{x \in W}^{} h^J_{x,w}(q) TC^J_w\,.$$ Then applying $e_J^-$ to both sides and using~\eqref{eq:pTC}, we get that \begin{equation}\label{eq:Pmin=h}
	P^{J^-}_{u,u'}(q) = h^J_{u,u'}(q) \qquad \hbox{ for } u,u' \in W^J.
\end{equation}  Thus, the parabolic Kazhdan-Lusztig polynomials are special cases of restriction coefficients.

% \newpage
		
\section{Examples of restriction coefficients}\label{sec:rescof}
In this section, we compute various examples of restriction coefficients. 

				\subsection{The elements $R_w$}
%				Using the quadratic relations, we get \begin{equation}\label{eq:TxCs}
	%					T_xC_s = \begin{cases}
		%						T_{xs} + qT_x & \hbox{ if } xs>x,
		%						\\
		%						T_{xs} + q^{-1}T_x & \hbox{ if } xs<x.
		%					\end{cases} 
	%				\end{equation}
%				
%				Let $h \in \cH$. Write $$h = \sum_{y \in W}^{} a_y q^{-\ell(y)} T_y.$$ Then \begin{align}
	%					hC_s &= \sum_{\substack{y \in W \\ ys>y}} \bigg( a_yq^{-\ell(y)} T_y + a_{ys}q^{-\ell(ys)}T_{ys} \bigg) C_s
	%					\nonumber
	%					\\
	%					&= \sum_{\substack{y \in W \\ ys>y}} \bigg( a_yq^{-\ell(y)} (T_{ys}+qT_y) + a_{ys}q^{-\ell(ys)}(T_y+q^{-1}T_{ys}) \bigg)
	%					\nonumber
	%					\\
	%					&= \sum_{\substack{y \in W \\ ys>y}} \bigg( (a_yq^{-\ell(ys)+1} + a_{ys}q^{-\ell(ys)-1}) T_{ys} + (a_yq^{-\ell(y)+1}+a_{ys}q^{-\ell(y)-1})T_y\bigg)
	%					\nonumber
	%					\\
	%					&= \sum_{\substack{y \in W \\ ys>y}} (qa_y+q^{-1}a_{ys}) q^{-\ell(y)}T_y + \sum_{\substack{y \in W \\ ys<y}} (q^{-1}a_y+qa_{ys})q^{-\ell(y)}T_y
	%					\label{eq:hC_s}
	%				\end{align}
%				
%				Therefore, if $hC_s = (q+q^{-1})h$, then \begin{align*}
	%					qa_y+q^{-1}a_{ys} &= (q+q^{-1}) a_y &\hbox{ if } ys>y,
	%					\\
	%					q^{-1}a_y+qa_{ys}&= (q+q^{-1})a_y &\hbox{ if }ys<y.
	%				\end{align*}
%				
%				Or, \begin{align*}
	%					a_{ys} = a_y.
	%				\end{align*}
%				
%				So, if for all $s \in J(\subseteq S)$, $hC_s = (q+q^{-1})h$, then for all $w \in W_J$, $a_y = a_{yw}$.

For $w\in W$, let \begin{equation}
	R_w = \sum_{y \leq w}^{} q^{\ell(w)-\ell(y)} T_y.
\end{equation} 

%				In https://nicolaslibedinsky.cl/wp-content/uploads/2021/04/affHeckecat-1.pdf  Libedinsky-Patimo $R_w$ is denoted $\mathbf{N}_w$.

%				Let $a_y = 1$ if $y \leq w$ and $0$ otherwise. Then by \eqref{eq:hC_s}, \begin{align*}
	%					R_wC_s &= q^{\ell(w)}\bigg( \sum_{\substack{y \in W \\ ys>y}}(qa_y+q^{-1}a_{ys})q^{-\ell(y)}T_y + \sum_{\substack{y \in W \\ ys<y}} (q^{-1}a_y+qa_{ys}) q^{-\ell(y)}T_y \bigg)
	%					\\
	%					&= q^{\ell(w)}\bigg( \sum_{\substack{y \in W \\ w\geq ys>y}}(q+q^{-1})q^{-\ell(y)}T_y + \sum_{\substack{y \in W \\ ys>y\geq w}}q\cdot q^{-\ell(y)}T_y 
	%					\\
	%					&\qquad\qquad\qquad+ \sum_{\substack{y \in W \\ w\geq y>ys}} (q^{-1}+q) q^{-\ell(y)}T_y + \sum_{\substack{y \in W \\ y>ys\geq w}}q\cdot q^{-\ell(y)}T_y\bigg)
	%				\end{align*}
%				
%				By lifting lemma, if $s \in S$, $ws<w$ then any $y \leq w$ satisfies $ys \leq w$. Then $R_wC_s = (q+q^{-1})R_w$ for such $w$ and $s$. And taking bar, $\overline{R}_wC_s = (q^{-1}+q)\overline{R}_w$.

%				For $w \in W$, \cite[\S 13.2]{BilleyLakshmibai} the following conditions are equivalent: \begin{enumerate}
	%					\item the KL polynomials all of the form $h_{y,w}(q) = q^{\ell(w)-\ell(y)}$,
	%					\item the KL polynomial $h_{\mathrm{id},w} = q^{\ell(w)}$,
	%					\item the variety $X(w)$ is rationally smooth at every point,
	%					\item the variety $X(w)$ is rationally smooth at identity,
	%					\item the Poincare polynomial of $X(w)$ is symmetric,
	%					\item the Bruhat graph $[\mathrm{id},w]$ is regular,
	%					\item (for type A) $w$ avoids the patterns $3412$ and $4231$.
	%				\end{enumerate}
%				

For $w \in W$, by \cite[\S 13.2]{BilleyLakshmibai}, $C_w = R_w$ if and only if the Schubert variety $X(w)$ is rationally smooth, and for type A, if and only if $w$ avoids the patterns $3412$ and $4231$. In particular, it applies for the longest element, Coxeter elements, and if $w$ is a product of pairwise commuting simple reflections.

\begin{remark}
	When $W = S_4$ is of type $A_3$, then the only two permutations for which $C_w|_{[1,2]}$ is not a single monomial is $w = s_2s_3s_1s_2 = (1,3)(2,4) = 3412$ and $s_1s_2s_3s_2s_1 = (1,4) = 4231$. So in this case $C_w|_{[1,2]}$ is a single monomial if $w$ avoids the patterns $3412$ and $4231$ (same as above). 
\end{remark}

%Now, \begin{align*}
%	R_w|_{J} = \sum_{\substack{y \leq w \\ y \in J}} q^{\ell(w)-\ell(y)}T_y
%\end{align*}
%
%In particular,  \begin{equation}\label{eq:R_w|JwhenwgeqwoJ}
%	R_w|_J = \sum_{y \in J}^{} q^{\ell(w)-\ell(y)}T_y = \begin{cases}
%		q^{\ell(w)-\ell(w_0^J)} R_{w_0^J}, \qquad \hbox{ for } w \geq w_0^J,
%		\\
%		R_w, \qquad \hbox{ if } w \leq w_0^J.
%	\end{cases}
%\end{equation} 

Equation \eqref{eq:rescoeffgenT} is true when $C$ is replaced by $R$ and $h$ by the corresponding transition matrix, since we have not used any property of $C$ except upper unitriangularity, then \begin{align}\label{eq:TuinvRwresgen}
	(T_{u^{-1}}R_w)|_J = \sum_{\substack{v \in W_J \\ uv \leq w}}^{} q^{\ell(w)-\ell(uv)}T_v = q^{\ell(w)-\ell(u)} \sum_{\substack{v \in W_J \\ uv \leq w}}^{} q^{-\ell(v)}T_v = q^{\ell(w)-\ell(u)} \sum_{\gamma \text{ max}} q^{-\ell(\gamma)} R_\gamma, 
\end{align} where the sum runs over all $\gamma$ which are maximal elements (w.r.t Bruhat order) of the set $\{v \in W_J \,|\, uv\leq w \}$.

%\begin{question}
%	How to write the KL basis elements in the $R$-basis? 
%\end{question}

%

\subsection{Dihedral groups}

For $m\in \ZZ_{\geq 3}\cup\{\infty\}$, the dihedral group $W$ is the Coxeter group of type $I_2(m)$ on two generators $s_1,s_2$ with relations $$ s_1^2=s_2^2=1, \qquad (s_1s_2)^m = 1 \text{ if } m<\infty.$$ 
%The group $W$ is finite or infinite depending on $m$ is finite or infinite. 
In this case \cite[page 163]{HumphreysCoxeterGroups} says \begin{align*}
	C_w = \sum_{x \leq w}^{} q^{\ell(w)-\ell(x)}T_x = R_w					
\end{align*}

%Therefore, by \eqref{eq:R_w|JwhenwgeqwoJ}, $$ C_w|_{[1]} = \begin{cases}
%	q^{\ell(w)-1}C_{s_1}, &\hbox{ if } w \geq s_1,
%	\\
%	q^{\ell(w)}, &\hbox{ if } w \leq s_2.
%\end{cases}$$

By \eqref{eq:TuinvRwresgen},

\begin{proposition}\label{prop:posDG}
	Let $W$ be of type $I_2(m)$ for $m \in \ZZ_{\geq 3} \cup \{\infty\}$ and $u\in W^{[1]}$ with $\ell(u) =k$. Then $$ (T_{u^{-1}}C_w)|_{[1]} = \begin{cases}
		q^{\ell(w)-(k+1)} C_{s_1} &\hbox{ if } w \geq us_1,
		\\
		q^{\ell(w)-k} &\hbox{ if }  u \leq w \leq (us_1)_{1 \leftrightarrow 2},
		\\
		0 &\hbox{ if } w \ngeq u,
	\end{cases}$$ where $(us_1)_{1 \leftrightarrow 2}$ denotes the element $us_1$ with $1$ and $2$ interchanged in the reduced expression.
\end{proposition}

\subsection{Parabolic Kazhdan-Lusztig polynomials for type $A$} We can compute all the parabolic Kazhdan-Lusztig polynomials for the Weyl group $W$ of type $A_n$ using \eqref{eq:Pmin=h} for the subset $J = [n-1]$ as follows. 

Since $W^{[n-1]} = \{1,s_n,s_{n-1}s_n,\ldots, s_1\ldots s_n\}$, and $C_{s_i\ldots s_n} = C_{s_i}\ldots C_{s_n} = R_{s_i\ldots s_n}$ for $i \in [n]$, we have that $C_u = R_u$ for $u \in W^{[n-1]}$. By \eqref{eq:Pmin=h}, the parabolic Kazhdan-Lusztig polynomials corresponding to sign representation are the elements $h^{[n-1]}_{u,u'}(q)$ for $u,u' \in W^{[n-1]}$. We can use \eqref{eq:TuinvRwresgen} to give a formula for these polynomials: \begin{equation}
	h^{[n-1]}_{u,u'}(q) = \begin{cases}
		1 &\hbox{ if } \ell(u) - \ell(u') = 0,
		\\
		q &\hbox{ if } \ell(u) - \ell(u') = 1,
		\\
		0 &\hbox{ otherwise.}
	\end{cases}
\end{equation} 

%				\subsection{$G_2$}
%				
%				The Coxeter system $G_2$ is isomorphic to the Coxeter system $I_2(6)$, which is the dihedral group of order $12$.
%				
%				\subsection{$H_3$}

%\subsection{Type $F_4$}
%
%In this case, $$ C_{s_3s_2s_3s_4s_3s_2s_3}|_{[1,2,3]} = (q+q^3) C_{s_3s_2s_3s_2} + q^2 C_{s_3},$$ which shows that not all coefficients are monomial. {\color{red}With our minimum computing abilities, we were unable to find an example for such in type $A$, where we checked upto type $A_5 = S_6$.} {\color{blue}For us SAGE worked for length up to $6$ only but we may have explicit (easy) calculation for parabolic KL polynomials for $J=[n-1]$ for any $n$.}

\subsection{Restriction coefficients for type A}\label{sec:resA}

In this section we investigate the $C$-expansion of the restriction  $(T_{u^{-1}}C_w)|_{[n-1]}$ for $W$ of type $A_n$, i.e, $W = S_{n+1}$. Here $u \in W^{[n-1]} = \{1,s_n,s_{n-1}s_n,\ldots, s_1\ldots s_n\}$ and  $W_{[n-1]} = \langle s_1,\ldots, s_{n-1}\rangle \cong S_n$. Recall the parabolic factorization notation $w = w^{[n-1]}\cdot w_{[n-1]}$ with $w^{[n-1]} \in W^{[n-1]}$ and $w_{[n-1]} \in W_{[n-1]}$. 
We have computed the restriction coefficients explicitly for $u \in \{s_1\ldots s_n, s_2\ldots s_n, s_3\ldots s_n\}$.

%\subsection{type $A$ coefficients}

 %%%%%%%%%%%%%%%%
 
 %%%%%%%%%%%%%%%%
The next lemma will be repeatedly used in the rest of this section. 
 
\begin{lemma}\label{lem:heq1}
	Let $y,x \in W_{[n-1]}$. Then for $1 \leq i \leq n$,
	$$h_{y,x}(q) = h_{s_i\ldots s_ny,s_i\ldots s_n x}(q).$$
\end{lemma}

\begin{proof}
	By lifting property \cite[Proposition 2.2.7]{BjornerBrenti}, for $y,x \in W_{[n-1]}$, $s_i\ldots s_ny < s_i \ldots s_nx$, if and only if $s_{i+1}\ldots s_ny < s_{i+1}\ldots s_nx$. Thus for $y,x \in W_{[n-1]}$, $s_i\ldots s_ny < s_i \ldots s_nx$, if and only if $y < x$. Also recall from \cite[Proposition 2.5.1]{BjornerBrenti}, the projection map $P^{[n-1]}: W \to W^{[n-1]}$ is Bruhat order preserving.
	
For $y,x \in W_{[n-1]}$ with $s_i\ldots s_ny < s_i\ldots s_n x$, we have $s_i(s_i\ldots s_n x) = s_{i+1}\ldots s_nx < s_i \ldots s_n x$ and $s_i(s_i\ldots s_n y) = s_{i+1}\ldots s_ny < s_i\ldots s_n x$, and since $P^{[n-1]}$ is order preserving, $s_i\ldots s_n y \nleq s_i(s_i\ldots s_nx) = s_{i+1}\ldots s_nx$. Hence by \autoref{prop:Casselmannrec1} with $s= s_i$, $$ h_{s_i\ldots s_ny,s_i\ldots s_nx} = h_{s_{i+1}\ldots s_ny,s_{i+1}\ldots s_nx}.$$ Continuing we get $$ h_{s_i\ldots s_ny,s_i\ldots s_nx} = h_{y,x}.$$
%	
%	For $y,x \in W_{[n-1]}$ with $y<x$, then $s_nx>x, s_nx > s_ny>y$ and $s_ny \nleq x$, since otherwise by applying $P^{[n-1]}$ we would get $s_n \leq 1$, a contradiction. Then by \autoref{prop:Casselmannrec1} $$h_{s_ny,s_nx}(q) = h_{y,x}(q).$$
%	Once again by applying $P^{[n-1]}$, we get   $s_nx<s_{n-1}s_nx, s_ny<s_{n-1}s_ny<s_{n-1}s_nx$ and  $s_{n-1}s_ny \nleq s_nx$  . Then by \autoref{prop:Casselmannrec1} $$ h_{s_{n-1}s_ny,s_{n-1}s_nx}(q) = h_{y,x}(q).$$
%	Continuing, we get $$ h_{y,x}(q) = h_{s_i\ldots s_n y, s_i\ldots s_n x}(q).$$
%	
%	{\color{red} The case when $y$ is not comparable to $x$?? Fine for $i=n$.}
\end{proof}	
			
\begin{proposition} For $w\in W$,
	$$ (T_{s_n\ldots s_1}C_w)|_{[n-1]} = \begin{cases}
		0 &\hbox{ if } w^{[n-1]} \neq s_1\ldots s_n,
		\\
		C_{w_{[n-1]}} &\hbox{ if } w^{[n-1]} = s_1\ldots s_n.
	\end{cases}$$
\end{proposition}

\begin{proof}
	By \autoref{lem:TuinvCwressupport}, if $w^{[n-1]} \neq s_1\ldots s_n$ then $(T_{s_n\ldots s_1}C_w)|_{[n-1]}=0$. Assume $w= s_1\ldots s_{n}\cdot v$ for some $v \in W_{[n-1]}$. Then by \eqref{eq:rescoeffgenT} and \autoref{lem:heq1} we get \begin{align*}
		(T_{s_n\ldots s_1}C_{s_1\ldots s_n\cdot v})|_{[n-1]} = \sum_{x \leq v}^{} h_{s_1\ldots s_n\cdot x,s_1\ldots s_n\cdot v}(q)T_x = \sum_{x \leq v}^{} h_{ x, v}(q)T_x = C_v. 
	\end{align*}
\end{proof}

By \eqref{eq:rescoeffgenT}, \begin{equation}\label{eq:si...sn}
	(T_{s_n\ldots s_i}C_{s_i\ldots s_n\cdot v})|_{[n-1]} = \sum_{x \leq v}^{} h_{s_i\ldots s_n\cdot x,s_i\ldots s_n\cdot v}(q)T_x = \sum_{x \leq v}^{} h_{x,v}(q)T_x = C_v.
\end{equation}

\begin{proposition} For $w\in W$,
	$$ (T_{s_n\ldots s_2}C_w)|_{[n-1]} = \begin{cases}
		C_{w_{[n-1]}} &\hbox{ if } w^{[n-1]} = s_2\ldots s_n,
		\\
		qC_{w_{[n-1]}} &\hbox{ if } w^{[n-1]} = s_1s_2\ldots s_n,
		\\
		0 &\hbox{ else.}
	\end{cases}$$
\end{proposition}

\begin{proof}
	When $w^{[n-1]} \ngeq s_2\ldots s_n$  then $(T_{s_n\ldots s_2}C_w)|_{[n-1]}=0$ by \autoref{lem:TuinvCwressupport}.
	When $w^{[n-1]} = s_2\ldots s_n$ then use \eqref{eq:si...sn}. When $w^{[n-1]} = s_1\ldots s_n$, then by \autoref{prop:Casselmannrecq} and \autoref{lem:heq1},
	$$ h_{s_2\ldots s_n\cdot y,s_1\ldots s_n\cdot x} = qh_{s_1\ldots s_n\cdot y, s_1\ldots s_n \cdot x} = q h_{y,x} \qquad \hbox{ for } y,x \in W_{[n-1]}.$$ By \eqref{eq:rescoeffgenT}, for $v \in W_{[n-1]}$, \begin{align*}
		&(T_{s_n\ldots s_2}C_{s_1\ldots s_n\cdot v})|_{[n-1]} = \sum_{y \in W_{[n-1]}}^{} h_{s_2\ldots s_n\cdot y,s_1\ldots s_n\cdot v}T_y 
		= q \sum_{y \in W_{[n-1]}}^{} h_{y,v}T_y = qC_v. 
	\end{align*} 
	
\end{proof}

For $y,x \in W_{[n-1]}$, by \autoref{prop:Casselmannrecq} and \autoref{lem:heq1}, we have
\begin{equation}
	h_{s_{i+1}\ldots s_ny,s_is_{i+1}\ldots s_nx} = qh_{s_i\ldots s_ny,s_i\ldots s_nx} = q h_{y,x}
\end{equation} and the same calculation as above shows that for $v \in W_{[n-1]}$, \begin{equation}
(T_{s_n\ldots s_{i+1}}C_{s_i\ldots s_n\cdot v} )|_{[n-1]} = qC_v.
\end{equation}

\begin{proposition}\label{prop:hi+2i}
	For $y,x \in W_{[n-1]}$ and $i \in [n-2]$,
	\begin{equation}
		h_{s_{i+2}\ldots s_ny,s_i\ldots s_nx} = \begin{cases}
			q^2h_{ y,x} &\hbox{ if } s_ix<x,
			\\
			qh_{ s_iy,x} + h_{y,x} &\hbox{ if } s_iy<y, s_ix>x,
			\\
			qh_{ s_iy,x} + q^{2}h_{y,x} &\hbox{ if } s_iy>y, s_ix>x.
		\end{cases}
	\end{equation}
	
\end{proposition}
\begin{proof}
	Using \autoref{prop:mainrecforKL} with $s= s_i$ we get
	\begin{align*}
		h_{s_{i+2}\ldots s_ny,s_i\ldots s_nx} &= h_{ s_{i+2}\ldots s_n s_iy,s_{i+1}\ldots s_nx} + q^c h_{s_{i+2}\ldots s_ny,s_{i+1}\ldots s_nx} 
		\\
		&\qquad- \sum_{\substack{w \\ s_{i+2}\ldots s_n y \leq w <s_{i+1}\ldots s_nx \\ s_iw<w}} \mu(w,s_{i+1}\ldots s_nx) h_{s_{i+2}\ldots s_ny,w}
		\\
		&= qh_{ s_iy,x} + q^{c+1} h_{y,x} 
		- \sum_{\substack{w \\ s_{i+2}\ldots s_n y \leq w <s_{i+1}\ldots s_nx \\ s_iw<w}} \mu(w,s_{i+1}\ldots s_nx) h_{s_{i+2}\ldots s_ny,w} 
	\end{align*} Where $c = \pm 1$ depending on $s_iy>y$ or $s_iy<y$.
	The last sum runs over $w = s_{i+2}\ldots s_n v$ for $v \in W_{[n-1]}$ with $s_iv < v$. Then \begin{align*}
		\mu(w,s_{i+1}\ldots s_nx) = [q] h_{s_{i+2}\ldots s_nv,s_{i+1}\ldots s_nx} = [q] q h_{v,x} = \begin{cases}
			0, &\hbox{ if } v \neq x,
			\\
			1 &\hbox{ if } v = x.
		\end{cases}
	\end{align*}
	Hence, \begin{align*}
		h_{s_{i+2}\ldots s_ny,s_i\ldots s_nx} &= qh_{ s_iy,x} + q^{c+1} h_{y,x} 
		- \chi_{s_ix<x}  h_{s_{i+2}\ldots s_ny,s_{i+2}\ldots s_nx} 
		\\
		&=qh_{ s_iy,x} + q^{c+1} h_{y,x} 
		- \chi_{s_ix<x}  h_{y,x} 
	\end{align*}
	If $s_iy<y$ and $s_ix<x$ then the final two terms cancel, and we are left with $qh_{s_iy,x} = q^2h_{y,x}$, where the last equality is by \autoref{prop:Casselmannrecq}. If $s_iy>y$ and $s_ix<x$ then $qh_{s_iy,x} = h_{y,x}$ by \autoref{prop:Casselmannrecq}, so the first and third terms cancel out. So we get \begin{equation}
		h_{s_{i+2}\ldots s_ny,s_i\ldots s_nx} = \begin{cases}
			q^2h_{ y,x} &\hbox{ if } s_ix<x,
			\\
			qh_{ s_iy,x} + h_{y,x} &\hbox{ if } s_iy<y, s_ix>x,
			\\
			qh_{ s_iy,x} + q^{2}h_{y,x} &\hbox{ if } s_iy>y, s_ix>x.
		\end{cases}
	\end{equation}
\end{proof}

\begin{corollary}
	For $y,x \in W_{[n-1]}$ and $i \in [n-2]$,	
	\begin{equation*}
		\mu(s_{i+2}\ldots s_ny,s_i\ldots s_nx) = \begin{cases}
			\delta_{s_iy,x} + \mu(y,x) &\hbox{ if } s_iy<y, s_ix>x,
			\\
			0 &\hbox{ otherwise. }
		\end{cases}
	\end{equation*}
\end{corollary}
\begin{proof}
	By taking the coefficient of $q$ in \autoref{prop:hi+2i}, we get \begin{align*}
		\mu(s_{i+2}\ldots s_ny,s_i\ldots s_nx) = \begin{cases}
			0 &\hbox{ if } s_ix<x,
			\\
			\delta_{s_iy,x} + \mu(y,x) &\hbox{ if } s_iy<y, s_ix>x,
			\\
			\delta_{s_iy,x} &\hbox{ if } s_iy>y, s_ix>x.
		\end{cases}
	\end{align*}  If $s_ix>x$ and $s_iy = x$ then $s_ix=y>x=s_iy$. So the last case will give $0$ always. 
\end{proof}

\begin{proposition}
	For $w \in W$,	\begin{align*}
		(T_{s_n\ldots s_3}C_w)|_{[n-1]} = \begin{cases}
			C_{w_{[n-1]}} &\hbox{ if } w^{[n-1]} = s_3\ldots s_n,
			\\
			qC_{w_{[n-1]}} &\hbox{ if } w^{[n-1]} = s_2\ldots s_n,
			\\
			q^2 C_{w_{[n-1]}} &\hbox{ if } w^{[n-1]} = s_1\ldots s_n \hbox{ and } s_1w_{[n-1]}<w_{[n-1]},
			\\
			q C_{s_1}C_{w_{[n-1]}} &\hbox{ if } w^{[n-1]} = s_1\ldots s_n \hbox{ and } s_1w_{[n-1]}>w_{[n-1]},
			\\
			0 &\hbox{ else.}
		\end{cases}
	\end{align*}
\end{proposition}

\begin{proof}
	The first two and the last cases follow from previous calculations. So we need to calculate when $w^{[n-1]} = s_1\ldots s_n$.
	When $y,x \in W_{[n-1]}$ and $s_1x<x$, then by \autoref{prop:Casselmannrecq} and \autoref{lem:heq1} we get $$ h_{s_3\ldots s_ny, s_1\ldots s_n x} = q h_{s_2\ldots s_ny,s_1\ldots s_nx} = q^2 h_{s_1\ldots s_ny,s_1\ldots s_nx} = q^2 h_{y,x}.$$ By \eqref{eq:rescoeffgenT} this gives the third case above. So the only remaining case to check is the fourth case.
Now,	
	\begin{align*}
		C_{s_1}C_x &= (T_{s_1}+q)C_x = (T_{s_1}+q)(\sum_{y\leq x}^{} h_{y,x}T_y) = \sum_{y \leq x}^{} h_{y,x}T_{s_1}T_y + q \sum_{y \leq x}^{} h_{y,x} T_y
		\\
		&= \sum_{\substack{y \leq x \\ s_1y>y}}^{} \bigg(h_{y,x} T_{s_1y} + h_{s_1y,x}(1+(q^{-1}-q)T_{s_1}) T_y\bigg) + q \sum_{y \leq x}^{} h_{y,x} T_y
		\\
		&=\sum_{\substack{y \leq x \\ s_1y>y}}^{} \bigg((h_{y,x}+h_{s_1y,x}(q^{-1}-q))T_{s_1y}+h_{s_1y,x}T_y\bigg) + q \sum_{y \leq x}^{} h_{y,x} T_y
		\\
		&= \sum_{\substack{y \leq x \\ s_1y>y}}^{} (h_{s_1y,x}+qh_{y,x}) T_y + \sum_{\substack{y \leq x \\ s_1y<y}}^{} (h_{s_1y,x}+(q^{-1}-q)h_{y,x}+qh_{y,x}) T_y
		\\
		&= \sum_{\substack{y \leq x \\ s_1y>y}}^{} (h_{s_1y,x}+qh_{y,x}) T_y + \sum_{\substack{y \leq x \\ s_1y<y}}^{} (h_{s_1y,x}+q^{-1}h_{y,x}) T_y.
	\end{align*}
	Thus the fourth case of the proposition is true if and only if for $x \in W_{[n-1]}$ and $s_1x>x$, \begin{align*}
		\sum_{y \in W_{[n-1]}}^{} h_{s_3\ldots s_ny,s_1\ldots s_nx}T_y = q\cdot\bigg(\sum_{\substack{y \leq x \\ s_1y>y}}^{} (h_{s_1y,x}+qh_{y,x}) T_y + \sum_{\substack{y \leq x \\ s_1y<y}}^{} (h_{s_1y,x}+q^{-1}h_{y,x}) T_y\bigg),
	\end{align*} or, for $y,x \in W_{[n-1]}$ with $s_1x>x$,
	\begin{align*}
		h_{s_3\ldots s_ny,s_1\ldots s_nx} = \begin{cases}
			qh_{s_1y,x}+q^2h_{y,x} &\hbox{ if } s_1y>y,
			\\
			qh_{s_1y,x}+h_{y,x} &\hbox{ if } s_1y<y,
		\end{cases}
	\end{align*}
which follows from \autoref{prop:hi+2i}.
	
\end{proof}

% \newpage
\section{A geometric proof of positivity of restriction coefficients}\label{sec:posGH}				

%We use \cite{Achar_book} for definitions regarding perverse sheaves, but with Soergel's normalization for Hecke algebras. 

%Since the paper of Grojnowski and Haiman is unpublished, 
For the sake of readability, we prove \autoref{thm:mainA} for finite Weyl groups following the proof of \autoref{th:GH} in \cite{GrojnowskiHaiman}.  It is useful to work with a different normalization of the Hecke algebra here, namely $$ T_s^2 = (q-1)T_s + q \qquad \hbox{ for } s \in S.$$ This is obtained from the earlier relation by replacing $T_s$ by $q^{-\frac{1}{2}} T_s$ for all $s \in S$. Here one needs to work with the coefficient ring $\ZZ[q^{\pm \frac{1}{2}}]$ and make corresponding changes to the definition, for example, the bar involution will take $q^{\pm \frac{1}{2}} \mapsto q^{\mp \frac{1}{2}}$. 

%{\color{red}We skip most of the background material which can be found in \cite{GrojnowskiHaiman}. As stated in \cite[\S 2.1]{GrojnowskiHaiman}, one only needs to worry about formal properties of mixed Hodge modules and functors between them, and their behaviour with respect to weights. 
%}

We work inside the setting of perverse sheaves and mixed Hodge modules, and use \cite{Achar_book} as our main reference. All geometric spaces considered here are quasiprojective complex algebraic varieties, in addition to their Zariski topology these varieties  are equipped with analytic topology. We will always consider sheaves with respect to the analytic topology. Any such variety is locally compact in the analytic topology, so there is no difficulty in using results from Chapter 1 of \cite{Achar_book}, for example using proper base change theorems.
For a quasiprojective algebraic variety $X$, let $D^b_c(X;\QQ)$ denote the derived category of bounded $\QQ$-complexes with constructible cohomologies. If $X$ has an action by an algebraic group $B$ then we can also consider the equivariant version $D^b_B(X;\QQ)$. A map $f:X \to Y$ of varieties induces maps on the derived categories: the derived pullback $f^*: D^b_c(Y) \to D^b_c(X)$ and the derived proper pushforward $f_!: D^b_c(X) \to D^b_c(Y)$. There is a $t$-structure on $D^b_c(X;\QQ)$ whose heart is the category $\mathrm{Perv}(X;\QQ)$ of perverse sheaves. If $Y \subset X$ is a smooth, connected, locally closed subvariety, and $\cL$ a local system of finite type on $Y$, then there is a intersection cohomology complex $\mathrm{IC}(Y,\cL) \in \mathrm{Perv}(X;\QQ)$. If $\cL$ is irreducible then $\mathrm{IC}(Y,\cL)$ is a simple object in $\mathrm{Perv}(X;\QQ)$, and any simple object in $\mathrm{Perv}(X;\QQ)$ is an intersection cohomology complex (\cite[Theorem 3.4.5]{Achar_book}). An object of $D^b_c(X;\QQ)$ is semisimple if it is isomorphic to a finite direct sum of shifts of simple perverse sheaves.   

We fix necessary notations and recall the categorification theorem \cite[Theorem 7.3.8]{Achar_book} which connects the Hecke algebra with geometry of flag varieties. Let $G$ be a connected reductive complex algebraic group and choose a Borel subgroup $B$ and a maximal torus $T \subset B$, $W$ be the corresponding Weyl group. The flag variety $\cF = G/B$ is a projective variety. For each $w \in W$, the open Schubert cells are $Y_w = BwB/B$, and these form a good stratification of $\cF$. There are natural inclusion maps $$i_w:Y_w \to \cF$$ for $w \in W$. The intersection cohomology complexes $\mathrm{IC}_w = \mathrm{IC}(Y_w;\QQ) \in D^b_c(\cF;\QQ)$ are $B$-equivariant. There is a convolution product that makes both the categories $\mathrm{Semis}_B(\cF;\QQ) \subset D^b_B(\cF;Q)$ monoidal, where $\mathrm{Semis}_B(\cF;\QQ)$ denotes the additive category of semisimple objects in $D^b_B(\cF;Q)$. The  categorification theorem says that the split Grothendieck group $\cK$ of $\mathrm{Semis}_B(\cF;\QQ)$ is isomorphic as rings with the Hecke algebra $\cH$ of $W$, where elements $[\mathrm{IC}_w]$ are mapped to the Kazhdan-Lusztig basis elements $C_w$. Under this isomorphism, the standard basis elements $T_w$ correspond to  $[i_{w!}\underline{\QQ}_{Y_w}]$, the classes of the derived proper pushforwards of the derived versions of constant sheaves. Thus, to show an element of the Hecke algebra is $C$-positive is equivalent to show that the corresponding element on the side of $\cK$ is arising from a semisimple element. 

Now we turn to our main goal in this section to show that for $w \in W$ and $u \in W^J$, the element $(T_{u^{-1}}C_w)|_J \in \cH_J$ is $C$-positive in the Kazhdan-Lusztig  basis of $\cH_J$. By \eqref{eq:leftrightconv} and \eqref{eq:res:CtoTCisres} it is equivalent to showing that for $w \in W$ and $u \in \! {}^JW$, $(C_wT_{u^{-1}})|_J \in \cH_J$ is $C$-positive. We choose to prove the later, since its more convenient to deal with left cosets than right cosets. We achieve our goal in two steps. 
In the first step, we lift the map of multiplying by $T_{u^{-1}}$ and restricting to $\cH_J$ at the level of functors between derived categories $D^b_c(\cF) \to D^b_c(\cF_L)$, where $\cF_L = L/B_L$, and $B_L= B\cap L$ is the Borel restricted to the Levi subgroup $L$ of $G$ corresponding to the subset $J \subset S$. Then in the second step, we  show that this lifted map preserves semisimplicity.  

For $J \subseteq S$, let $P$ and $L$ be the parabolic and Levi subgroups of $G$  corresponding to $J$ respectively. Then $$ P = LB = \underset{w \in W_J}{\bigcup} BwB.$$
There are natural inclusion maps $$ i_w: Y_w \xhookrightarrow{} \cF \qquad\hbox{ and }\qquad i_v^J: Y^J_v \xhookrightarrow{} \cF_L,$$ 
where $v\in W_J$ and $Y^J_v = B_LvB_L/B_L$.
Now fix $x \in \! {}^JW$. We suppress the dependence on the element $x$ in the notation of the maps given below.
There is an isomorphism of varieties $$ \varphi: \cF_L \to LxB/B \quad\hbox{ given by }\quad gB_L \mapsto gxB.$$ The well-definedness of $\varphi$ comes from the fact that $\mathrm{stab}(xB)\cap L = B_L$, since for $b_L,b'_L \in L$, \begin{align*}
	b_LB _L = b'_L B_L \iff  {b'_L}^{-1}b_L \in B_L \iff {b'_L}^{-1}b_L xB = xB \iff b_L xB = b'_L xB. 
\end{align*} which proves injectivity as well as well-defined. The map $L \to Lx$ given by $g \mapsto gx$ is surjective, so $\varphi$ is surjective as well.

\noindent
Since $P=LB$ we get a map $$\pi: PxB/B \to LxB/B \quad \hbox{ given by } \quad (g_L\cdot b)xB \mapsto g_LxB.$$ It is well-defined since $\mathrm{stab}(xB) \cap P \subseteq B$. 

\noindent
Finally, there is inclusion $$j: PxB/B \xhookrightarrow{} G/B = \cF.$$

The next proposition realizes the map $h \mapsto (hT_{x^{-1}})|_J$ from $\cH \to \cH_J$ as a functor.

\begin{proposition} For $x\in\! {}^JW$ and notations as above
	$$ \varphi^* \pi_! j^* (i_w)_! \underline{\QQ}_{Y_w} \cong \begin{cases}
		(i^J_{v})_! \underline{\QQ}_{Y^J_{v}} & \hbox{ if } wx^{-1} = v \in W_J,
		\\
		0 &\hbox{ otherwise.}
	\end{cases} $$
\end{proposition}

\begin{proof}
	Suppose $wx^{-1} \in W_J$. Then, since $$PxB/B = \underset{v \in W_J}{\bigcup}  Y_{vx},$$ we have the following cartesian square
	\[\begin{tikzcd}
		{Y_w} && {Y_w} \\
		\\
		{PxB/B} && {\cF}
		\arrow["{j' = \mathrm{id}}", from=1-1, to=1-3]
		\arrow["{i'}"', from=1-1, to=3-1]
		\arrow["i_w", from=1-3, to=3-3]
		\arrow["j"', from=3-1, to=3-3]
	\end{tikzcd}\]
	In this case by base change \cite[Theorem 1.2.13]{Achar_book} and pullback of constant sheaf \cite[(1.2.6)]{Achar_book} (note that $(j')^*$ is the identity functor) we have $$ j^* (i_w)_! \underline{\QQ}_{Y_w} \cong  i'_! (j')^*\underline{\QQ}_{Y_w} = i'_! \underline{\QQ}_{Y_w}. $$
	So we have to compute $$\varphi^* \pi_! j^* (i_w)_! \underline{\QQ}_{Y_w} \cong \varphi^* \pi_!i'_! \underline{\QQ}_{Y_w}\cong \varphi^* (\pi\circ i')_! \underline{\QQ}_{Y_w} .$$
Here in the second isomorphism, we are using the fact that the spaces involved are locally compact, since $_!$ does not always factor with compositions.
	
	Now consider the following commutative diagram: 	%[TODO: prove cartesian etc. here again]
	% https://q.uiver.app/#q=WzAsNCxbMCwwLCJZXkpfdiJdLFsyLDAsIllfdyJdLFsyLDIsIkx4Qi9CIl0sWzAsMiwiXFxtYXRoY2Fse0Z9X0wiXSxbMCwxLCJcXHZhcnBoaSciXSxbMywyLCJcXHZhcnBoaSIsMl0sWzEsMiwiXFxwaVxcY2lyYyBpJyJdLFswLDMsImleSl92IiwyXV0=
	\[\begin{tikzcd}
		{Y^J_v} && {Y_w} \\
		\\
		{\mathcal{F}_L} && {LxB/B}
		\arrow["{\varphi'}", from=1-1, to=1-3]
		\arrow["{i^J_v}"', from=1-1, to=3-1]
		\arrow["{\pi\circ i'}", from=1-3, to=3-3]
		\arrow["\varphi"', from=3-1, to=3-3]
	\end{tikzcd}\]
	
	\noindent The map $\varphi':Y^J_v = B_LvB_L/B_L \to Y_w = BwB/B$ is given by $b_LvB_L \mapsto b_LvxB$, which is in fact the restriction of $\varphi$. We show that the above commutative diagram is a cartesian square. It is enough to show that the map $$ \psi: Y^J_v \to \mathcal{F}_L\times_{LxB/B} Y_w \text{ defined as } \psi(g_LvB_L)=(i_v^J(g_LvB_L), \varphi'(g_LvB_L))$$
	is an isomorphism. %We have show that $\psi_2: \mathcal{F}_L\times_{LxB/B} Y_w \to Y^J_v$
	
	\noindent As $\varphi'$ (or $i_v^J$) is injective, so $\psi$ is injective as well.  Let us now show the surjectivity of $\psi$. Let $(\ell B_L, bwB)\in \mathcal{F}_L\times_{LxB/B} Y_w$. Then by the definition of $\mathcal{F}_L\times_{LxB/B} Y_w$, we have $$\varphi(\ell B_L)=\pi\circ i'(bwB).$$ This gives us that $\ell B_L=\ell' vB_L$, where $\ell' \in L$ such that $b=\ell' u$ and $u$ belongs to the unipotent radical $R_u(P)$ of $P$. Next we argue that $\psi(\ell' vB_L)=(\ell B_L, bwB)$, i.e. to show $bwB= \ell'vxB$, which is equivalent to showing $w^{-1}uw\in B$. Now as $w\in W_J$ and $u\in R_u(P)$ and $R_u(P)\unlhd P$, so $w^{-1}uw\in R_u(P)\subset B$.
	
	\noindent Again, using base-change and pullback of constant sheaf $$ \varphi^* (\pi\circ i')_! \underline{\QQ}_{Y_w} \cong (i^J_v)_! {(\varphi')}^* \underline{\QQ}_{Y_w} = (i^J_v)_! \underline{\QQ}_{Y^J_v} .$$
	
	If $wx^{-1} \notin W_J$ then the intersection of $Y_w$ and $PxB/B$ is empty, so in that case we need to replace $Y_w$ in the first commutative square above by $\emptyset$. In this case base change will give $0$ since constant sheaf on empty set is $0$.
\end{proof}
	%	So let's continue with $w = vx$, $v \in W_J$. 	

This proves that under the isomorphism $\mathcal{K} \to \cH$,  \begin{equation}
	\varphi^*\pi_! j^*(h) = (hT_{x^{-1}})|_J \qquad\hbox{for} \ h \in \cH.
\end{equation}	

%Thus we have the commutative diagram 

For the final step we show that \begin{proposition}
	For $x\in\! {}^JW$ and notations as above the functor $\varphi^* \pi_! j^*: D^b_c(\cF) \to D^b_c(\cF_L)$ takes semisimple complexes to semisimple complexes.
\end{proposition}

\begin{proof}
	From \cite[\S 5.6]{Achar_book}, given a smooth complex variety $X$ there is a finite-length abelian category $\mathrm{MHM}(X;\QQ)$ of mixed Hodge modules. It comes equipped with an exact, faithful functor $$ \mathrm{rat}: \mathrm{MHM}(X;\QQ) \to \mathrm{Perv}(X;\QQ).$$ This gives a functor of derived categories $$ \mathrm{rat}: D^b\mathrm{MHM}(X;\QQ) \to D^b(\mathrm{Perv}(X;\QQ)) \cong D^b_c(X;\QQ),$$ where the last equivalence is Beilinson's theorem \cite[Theorem 4.5.9]{Achar_book}. There is a weight filtration on $\mathrm{MHM}(X;\QQ)$ and this gives a concept of weights on $D^b\mathrm{MHM}(X;\QQ)$. Purity is defined by objects which are concentrated on a single weight. If $\cF \in D^b\mathrm{MHM}(X;\QQ)$ is pure, then the object $\mathrm{rat}(\cF)\in D^b_c(X;\QQ)$ is semisimple \cite[Theorem 5.6.24]{Achar_book}. 
	
	%To prove the proposition, it is enough to prove that the lifted map $\varphi^*\pi_!j^*: D^b\mathrm{MHM}(\cF;\QQ) \to D^b\mathrm{MHM}(\cF_L;\QQ)$ preserves purity.
	
	By \cite{Saito_MHM}, we can lift the intersection cohomology complexes $\mathrm{IC}_w$ to objects in $D^b\mathrm{MHM}(\cF;\QQ)$. In this setting the functors also get lifted appropriately. To show that $\varphi^*\pi_! j^* \mathrm{IC}_w$ is a semisimple object of $D^b_c(\cF_L;\QQ)$, we can show that the corresponding object in the (derived category of) mixed Hodge modules setting is pure. This is obtained by using Braden's hyperbolic localization which we state now. 
	
	 Let $X$ be a quasi-projective variety with a $\CC^\times$-action, let $Z$ be a connected component of $X^{\CC^\times}$ and $Y$ be the attracting variety to $Z$ with attracting map $\pi:Y\to Z$. Then \cite[Theorem 1, Eq. (1) and Theorem 8]{Braden_hyperboliclocalization} says that the hyperbolic localization functor $\pi_! i^*$ preserves purity of weakly equivariant mixed sheaves. 
	
	 There is an action of $\CC^\times$ on $\cF$ such that the connected components of the fixed locus are $LxB/B$ such that $x \in \JW$. The attracting variety to $LxB/B$ is $PxB/B$ with attracting map $\pi$. In our setting we can apply hyperbolic localization with $X = \cF, Y = PxB/B$ and $Z = LxB/B$. Since $\varphi$ is isomorphism it is smooth and in that case $\varphi^*$ will preserve purity and weight (\cite[Corollary 5.6.20]{Achar_book}), hence semisimplicity. Then by applying $\mathrm{rat}$, $\varphi^*\pi_!i^*$ takes $\mathrm{IC}$ sheaves to semisimple objects in $D^b_c(\cF_L;\QQ)$.  
	
%	Since $\varphi$ is isomorphism it is smooth (\cite[Appendix D]{GortzWedhorn1}) and in that case $\varphi^*$ will preserve purity and weight (\cite[Corollary 5.6.20]{Achar_book}), hence semisimplicity. Thus the composition $\varphi^*\pi_!i^*$ preserves semisimplicity.
\end{proof}

			% \newpage
%			\section{Leftover}
%			
%			\subsection{}From GH: $h^J_{uv,w}$ is actually the parabolic KL polynomial when $uv, w \in W^J$. Is the restriction interpretation there in the literature? Is the factorization property of KL matrix there in the literature?
%			
%			\subsection{} We can do the hybrid basis construction with $C$ replaced by $\overline{T}$ and that gives a factorization of the $R$-matrix: what good does that do?
%			
%			\subsection{} Can we compute the $C$-expansion of $(T_{u^{-1}}C_{u'})|_J$ for $u,u' \in W^J$? Then this can serve as a initial condition for the recursion for general type. For type A we know because $C_u = R_u$ there.
%			
%			\section{TODOs}
%			\subsection{} Write a short exposition explaining why parabolic KL is a special case of restriction coefficients following GH.
%			
%			\subsection{} Find out where/if parabolic KL has coefficient is monomial property (when $J = [n-1]$ in type $A_n$).
%			
%			\subsection{} Check coefficient is monomial property for $A_5$ (checked for length $\leq 9$).
%			
%			\subsection{} What are restriction coefficients when $J = [n]\setminus \{i\}$ where $i$ is not $1$ or $n$?
						\bibliographystyle{alpha}
						\bibliography{main}
						
					\end{document}